\documentclass{amsart}

\usepackage[english]{babel}
\usepackage{hyperref}
\usepackage{xcolor}
\usepackage{braket}
\usepackage{relsize}

\newtheorem{theorem}{Theorem}[section]
\newtheorem{lemma}[theorem]{Lemma}
\newtheorem{proposition}[theorem]{Proposition}
\newtheorem{corollary}[theorem]{Corollary}
\theoremstyle{definition}
\newtheorem{definition}[theorem]{Definition}
\newtheorem{example}[theorem]{Example}

\numberwithin{equation}{section}
\begin{document}

\title[Interpolation between domains of powers of operators]{Interpolation between domains of powers of operators in quaternionic Banach spaces}

\author{F. Colombo}
\address{(FC) Dipartimento di Matematica, Via E. Bonardi 9, 20133 Milano, Italy}
\email{fabrizio.colombo@polimi.it}

\author{P. Schlosser}
\address{(PS) Dipartimento di Matematica, Via E. Bonardi 9, 20133 Milano, Italy}
\email{pschlosser@math.tugraz.at}
\thanks{P. Schlosser was funded by the Austrian Science Fund (FWF) under Grant No. J 4685-N and by the European Union -- NextGenerationEU}

\subjclass[1991]{47A10,47A60}

\begin{abstract}
In contrast to the classical complex spectral theory, where the spectrum is related to the invertibility of $\lambda-A:D(A)\subseteq X_\mathbb{C}\rightarrow X_\mathbb{C}$, in the noncommutative quaternionic $S$-spectral theory one uses the invertibility of the second order polynomial $Q_s(T):=T^2-2\text{Re}(s)T+|s|^2:D(T^2)\subseteq X\rightarrow X$ to define the $S$-spectrum, where $X$ is a quaternionic Banach space. In this paper we will consider quaternionic operators $T$, for which at least one ray $\{te^{i\omega}\;|\;t>0\}$, $\omega\in[0,\pi]$, $i\in\mathbb{S}$ is contained in the $S$-resolvent set, and the inverse operator $Q_s^{-1}(T)$ admits certain decay properties on this ray. Utilizing the $K$-interpolation method, we then demonstrate that the domain $D(T^k)$ of the $k$-th power of $T$ is an intermediate space between $D(T^n)$ and $D(T^m)$, whenever $n<k<m\in\mathbb{N}_0$. Moreover, also a characterization of the interpolation space $(X,D(T^n))_{\theta,p}$, $\theta\in(0,1)$, $p\in[1,\infty]$, in is given in terms of integrability conditions on the pseudo $S$-resolvent $Q_s^{-1}(T)$.
\end{abstract}

\maketitle

\section{Introduction}

The foundation of quaternionic quantum mechanics, see \cite{BF}, was the initial motivation to develop the spectral theory on the $S$-spectrum, see \cite{FJBOOK,CGK,ColomboSabadiniStruppa2011}. The central challenge in advancing the spectral theory for both quaternionic and Clifford algebras settings has revolved around defining a suitable concept of spectrum. This definition needs to be applicable for the spectral theorem as well as the hyperholomorphic functional calculus within this framework.

Given the existence of multiple possible definitions of hyperholomorphicity in the hypercomplex setting, the difficulty has been in establishing the most pertinent notion that serves the purpose. The identification of slice hyperholomorphic functions played a pivotal role in determining the $S$-spectrum for both quaternionic and Clifford operators. Importantly, this identification was achieved solely through methods in hypercomplex analysis, rather than relying on physical motivations or operator theory considerations despite the foundational motivation stemming from quaternionic quantum mechanics, see the introduction in \cite{CGK} for all the details on the discovery of the $S$-spectrum.

A key distinction between the classical spectral theory in complex Banach spaces $X_\mathbb{C}$ and the $S$-spectral theory on quaternionic Banach spaces $X$, lies in the structure of the corresponding resolvent operator. While in the complex case the spectrum is related to the invertibility of the operator $\lambda-A:D(A)\subseteq X_\mathbb{C}\rightarrow X_\mathbb{C}$, one has to consider the second order polynomial
$T^2-2s_0T+|s|^2:D(T^2)\subseteq X\rightarrow X$, in the quaternionic case, in order to preserve holomorphicity of the resulting resolvent operator. Note, that is essential to consider the domain of $T^2$, and the majority of considerations related to the spectral theory on the $S$-spectrum rely on the domain of $T^2$ and not directly on $T$ itself.

In the spectral theory on the $S$-spectrum, there is an additional interest in exploring fine structures on the $S$-spectrum introduced in \cite{ACQS2016,CDPS1,Fivedim,Polyf1,MPS23}. These investigations involve the functional calculus for sectorial operators associated with axially harmonic functions and axially polyanalytic functions. For an alternative approach, based on monogenic function theory \cite{DSS92}, of the $H^\infty$-functional calculus is addressed in \cite{TAOBOOK}. The significance of sectorial operators and their $H^\infty$-functional calculus is fundamental to numerous problems in operator theory, as detailed in references such as \cite{MC97,MC06} and \cite{Haase}.

\medskip

Interpolation spaces have been extensively studied in the literature and classical references include the monographs \cite{BERG_INTER,BRUDNYI_INTER,TRIEBEL}. These spaces find applications across various fields, particularly in partial differential equations, with an extensive list of applications mentioned in \cite{Ale_PDE} and the references therein. The characterization of interpolation spaces in terms of the domains of sectorial operators is particularly relevant, as outlined in \cite{Ale_INTER}.

\medskip

\textit{The aim of this paper} is to generalize two important results of the well established $K$-interpolation theory in complex Banach spaces, to the quaternionic Banach space.

\medskip

In order to introduce our first result, we mention that in the complex case the interpolation space between $X_\mathbb{C}$ and $D(A)$ can be characterized by
\begin{equation}\label{Eq_X_DA_interpolation}
(X_\mathbb{C},D(A))_{\theta,p}=\Set{x\in X_\mathbb{C} | \lambda\mapsto\lambda^\theta\varphi_x(\lambda)\in L^p_*(0,\infty)},
\end{equation}
using the function $\varphi_x(\lambda):=\Vert A(\lambda-A)^{-1}x\Vert$ and the weighted Lebesgue space $L^p_*(0,\infty):=L^p((0,\infty),dt/t)$, see \cite[Proposition 3.1.1]{Ale_INTER} and \cite[Proposition 3.1.6]{Ale_INTER} for a similar result for $(X_\mathbb{C},D(A^2))_{\theta,p}$. Our result in Theorem \ref{thm_X_DTn_characterization} generalizes \eqref{Eq_X_DA_interpolation} to the quaternionic $S$-spectral setting, i.e. replaces the resolvent $(\lambda-A)^{-1}$ by the $S$-pseudo resolvent $Q_s^{-1}(T)=(T^2-2s_0T+|s|^2)^{-1}$, and we also consider domains $D(T^n)$ of arbitrary powers. In particular, we establish the characterization
\begin{equation*}
(X,D(T^n))_{\theta,p}=\Set{x\in X | t\mapsto t^{n\theta}\psi_x(t)\in L^p_*(0,\infty)},
\end{equation*}
where it turns out that, due to the second order pseudo $S$- resolvent operator $Q_s^{-1}(T)$, the function $\varphi_x$ has to be replaced by $\psi_x$ which has to be defined for even and for odd values of $n$ differently, namely
\begin{equation*}
\psi_x(t):=\begin{cases} \Vert T^nQ_{te^{i\omega}}^{-\frac{n}{2}}(T)x\Vert, & n\text{ even}, \\ \Vert T^{n+1}Q_{te^{i\omega}}^{-\frac{n+1}{2}}(T)x\Vert+t\Vert T^nQ_{te^{i\omega}}^{-\frac{n+1}{2}}(T)x\Vert, & n \text{ odd.} \end{cases}
\end{equation*}
The second goal of this paper deals with intermediate spaces, these are well known in the complex case, but the extension to the quaternionic setting requires more sophisticated proofs. In Theorem \ref{thm_J_intermediate} and Theorem \ref{thm_K_intermediate} we prove that for every $n<k<m\in\mathbb{N}_0$, the domain $D(T^k)$ is an intermediate space between $D(T^n)$ and $D(T^m)$. More precisely, using the Definition \ref{defi_Intermediate_spaces} of intermediate spaces we prove that
\begin{equation*}
D(T^k)\in J_{\frac{k-n}{m-n}}\big(D(T^n),D(T^m)\big)\qquad\text{and}\qquad D(T^k)\in K_{\frac{k-n}{m-n}}\big(D(T^n),D(T^m)\big).
\end{equation*}
Explicitly this means that there exist constants $C_J,C_K\geq 0$, such that
\begin{align*}
\Vert x\Vert_{D(T^k)}&\leq C_J\Vert x\Vert_{D(T^n)}^{\frac{m-k}{m-n}}\Vert x\Vert_{D(T^m)}^{\frac{k-n}{m-n}}, && x\in D(T^m), \\
K(t,x)&\leq C_Kt^{\frac{k-n}{m-n}}\Vert x\Vert_{D(T^k)}, && x\in D(T^k),\,t>0,
\end{align*}
where $K(t,x)$ is the $K$-functional related to the spaces $D(T^n)$, $D(T^m)$.

\section{Interpolation theory of quaternionic Banach spaces}\label{sec_Interpolation_theory_of_quaternionic_Banach_spaces}

In this section we will give standard results of interpolation theory, which are well known in the complex case, but generalizes also to Banach spaces over the quaternions. We start with a short introduction to quaternionic numbers and quaternionic (right) Banach spaces.

\medskip

The \textit{quaternionic numbers} are defined as
\begin{equation*}
\mathbb{H}:=\Set{s_0+s_1e_1+s_2e_2+s_3e_3 | s_0,s_1,s_2,s_3\in\mathbb{R}},
\end{equation*}
with the three imaginary units $e_1,e_2,e_3$ satisfying the relations
\begin{equation*}
e_1^2=e_2^2=e_3^2=-1\qquad\text{and}\qquad\begin{array}{l} e_1e_2=-e_2e_1=e_3, \\ e_2e_3=-e_3e_2=e_1, \\ e_3e_1=-e_1e_3=e_2. \end{array}
\end{equation*}
For every quaternion $s\in\mathbb{H}$, we set
\begin{align*}
\text{Re}(s)&:=s_0, && (\textit{real part}) \\
\text{Im}(s)&:=s_1e_1+s_2e_2+s_3e_3, && (\textit{imaginary part}) \\
\overline{s}&:=s_0-s_1e_1-s_2e_2-s_3e_3, && (\textit{conjugate}) \\
|s|&:=\sqrt{s_0^2+s_1^2+s_2^2+s_3^2}. && (\textit{modulus})
\end{align*}
The unit sphere of purely imaginary quaternions is defined as
\begin{equation*}
\mathbb{S}:=\Set{s\in\mathbb{H} | s_0=0\text{ and }|s|=1},
\end{equation*}
and for every $i\in\mathbb{S}$ we consider the complex plane
\begin{equation*}
\mathbb{C}_i:=\Set{x+iy | x,y\in\mathbb{R}},
\end{equation*}
which is an isomorphic copy of the complex numbers, since every $i\in\mathbb{S}$ satisfies $i^2=-1$. Moreover, for every $s\in\mathbb{H}$ we consider the corresponding $2$-sphere
\begin{equation*}
[s]:=\Set{\text{Re}(s)+i\vert\text{Im}(s)\vert | i\in\mathbb{S}}.
\end{equation*}

\begin{definition}[Right quaternionic Banach space]
A \textit{quaternionic right vector space} is an additive group $(X,+)$ equipped with a scalar multiplication from the right $\cdot:X\times\mathbb{H}\rightarrow X$, which satisfies
\begin{equation*}
(x+y)s=xs+ys,\qquad x(s+q)=xs+sq,\qquad x(sq)=(xs)q,\qquad x1=x,
\end{equation*}
for every $x,y\in X$, $s,q\in\mathbb{H}$. A \textit{right norm} on $X$ is a mapping $\Vert\cdot\Vert:X\rightarrow[0,\infty)$, with the properties

\begin{enumerate}
\item[i)] $\Vert xs\Vert=\Vert x\Vert\,|s|$,\hspace{1.6cm} $x\in X,\,s\in\mathbb{H}$,
\item[ii)] $\Vert x+y\Vert\leq\Vert x\Vert+\Vert y\Vert$,\qquad $x,y\in X$,
\item[iii)] $\Vert x\Vert=0\Rightarrow x=0$,\hspace{1.15cm} $x\in X$.
\end{enumerate}
Finally, $X$ is called \textit{right Banach space}, if it is complete with respect to this norm.
\end{definition}

On right Banach spaces we can now define right linear operators.

\begin{definition}[Right linear operator]
Let $X$ be a right Banach space. An operator $T:X\rightarrow X$ is called right linear, if its domain $D(T)\subseteq X$ is a right linear subspace of $X$ and for every $x,y\in D(T)$, $s\in\mathbb{H}$ there holds
\begin{equation*}
T(x+y)=T(x)+T(y)\qquad\text{and}\qquad T(xs)=T(x)s.
\end{equation*}
For any right linear operator $T$ and any real number $\lambda$, we define the operators
\begin{equation*}
(\lambda T)(x):=(T\lambda)(x):=T(x)\lambda\qquad\text{with }D(\lambda T):=D(T\lambda):=D(T).
\end{equation*}
The operator $T$ is called closed, if $D(T)$ is complete with respect to the norm
\begin{equation*}
\Vert x\Vert_{D(T)}:=\Vert x\Vert+\Vert Tx\Vert,\qquad x\in D(T).
\end{equation*}
\end{definition}

\begin{definition}[Interpolation couple]
Two right Banach spaces $X$ and $Y$ are called \textit{interpolation couple} if there exists some topological Hausdorff space $Z$, such that both $X$ and $Y$ are continuously embedded in $Z$.
\end{definition}

For any interpolation couple $X,Y$, the intersection $X\cap Y$ as well as the sum $X+Y$ are right Banach spaces with norms
\begin{align*}
\Vert x\Vert_{X\cap Y}:=&\max\{\Vert x\Vert_X,\Vert x\Vert_Y\}, \\
\Vert x\Vert_{X+Y}:=&\inf\limits_{x=a+b,a\in X,b\in Y}\big(\Vert a\Vert_X+\Vert b\Vert_Y\big).
\end{align*}
The proof is the same as in the complex case \cite[Lemma 2.3.1]{BERG_INTER}. In between $X\cap Y$ and $X+Y$ we will now construct several interpolation spaces. A key role will play the following $K$-functional, which is a family of equivalent norms on $X+Y$.

\begin{definition}[$K$-functional]
For an interpolation couple $X,Y$ we define the \textit{$K$-functional}
\begin{equation}\label{Eq_K_functional}
K(t,x):=\inf_{x=a+b,a\in X,b\in Y}(\Vert a\Vert_X+t\Vert b\Vert_Y),\qquad t>0,\,x\in X+Y.
\end{equation}
\end{definition}

Clearly, the $K$-functional depends on the spaces $X$ and $Y$, but for simplicity we will omit this dependence. In the following we will need for any $p\in[1,\infty]$ the space
\begin{equation}\label{Eq_Lpstar}
L^p_*(0,\infty):=L^p\big((0,\infty),dt/t\big),
\end{equation}
which is the Lebesgue-space on $(0,\infty)$ with the weighted measure $dt/t$. One can easily check by substituting $t\rightarrow t^\alpha$, that with the convention $|\alpha|^{\frac{1}{\infty}}:=1$, this space has the property
\begin{equation}\label{Eq_Lpstar_symmetry}
\Vert f(t)\Vert_{L^p_*(0,\infty)}=|\alpha|^{\frac{1}{p}}\Vert f(t^\alpha)\Vert_{L^p_*(0,\infty)},\qquad\alpha\in\mathbb{R}\setminus\{0\}.
\end{equation}
With the space \eqref{Eq_Lpstar} and the $K$-functional \eqref{Eq_K_functional}, we now define interpolation spaces.

\begin{definition}[Interpolation spaces]
Let $X,Y$ be an interpolation couple. Then for every $0<\theta<1$ and $p\in[1,\infty]$, we define the \textit{interpolation space}
\begin{equation}\label{Eq_Interpolation_space}
(X,Y)_{\theta,p}:=\Set{x\in X+Y | t\mapsto t^{-\theta}K(t,x)\in L^p_*(0,\infty)}.
\end{equation}
Equipped with the right norm
\begin{equation*}
\Vert x\Vert_{\theta,p}:=\Vert t^{-\theta}K(t,x)\Vert_{L^p_*(0,\infty)},
\end{equation*}
this space becomes a right Banach space.
\end{definition}

That $\Vert\cdot\Vert_{\theta,p}$ is a right norm, follows immediately from the fact that $x\mapsto K(t,x)$ as well as $\Vert\cdot\Vert_{L^p_*(0,\infty)}$ are right norms, see \cite[Lemma 3.1.1]{BERG_INTER}. For the completeness of $(X,Y)_{\theta,p}$ one can imitate the proof of \cite[Proposition 1.1.5]{Ale_INTER}. These interpolation spaces admit now certain immediate properties. Some of them will be \textit{continuous embeddings}, where we use the notation that $V\hookrightarrow W$, whenever $V\subseteq W$ and there exists some constant $C\geq 0$ such that $\Vert v\Vert_W\leq C\Vert v\Vert_V$, for every $v\in V$.

\begin{lemma}\label{lem_Properties_interpolation_spaces}
Let $X,Y$ be an interpolation couple, $\theta,\eta\in(0,1)$, $p,q\in[1,\infty]$. Then there holds:

\begin{enumerate}
\item[i)] $(X,Y)_{\theta,p}=(Y,X)_{1-\theta,p}$, with equal norms.
\item[ii)] $(X,X)_{\theta,p}=X$, with equivalent norms.
\item[iii)] $(X,Y)_{\theta,p}=\{0\}$,\hspace{4.75cm} if $X\cap Y=\{0\}$.
\item[iv)] $X\cap Y\hookrightarrow(X,Y)_{\theta,p}\hookrightarrow(X,Y)_{\theta,q}\hookrightarrow X+Y$,\qquad if $p\leq q$.
\item[v)] $(X,Y)_{\theta,p}\hookrightarrow(X,Y)_{\eta,q}$,\hspace{3.75cm} if $Y\hookrightarrow X$ and $\theta>\eta$.
\item[vi)] $X\cap Y$ is dense in $(X,Y)_{\theta,p}$,\hspace{2.85cm} if $p<\infty$.
\end{enumerate}
\end{lemma}

\begin{proof}
The proofs are the formally the same as in the complex case. The statements i) -- iv) can be found in \cite[Proposition 1.1.3]{Ale_INTER}, the property v) is proven in \cite[Proposition 1.1.4]{Ale_INTER} and the result vi) can be taken from \cite[Proposition 1.2.5]{Ale_INTER}.
\end{proof}

Next, we consider the boundedness of operators restricted to interpolation spaces.

\begin{theorem}
Let $X,Y$ and $V,W$ be two interpolations couples, $\theta\in(0,1)$ and $p\in[1,\infty]$. Consider some right linear operator $T:X+Y\rightarrow V+W$, with bounded restrictions $T|_X\in\mathcal{B}(X,V)$ and $T|_Y\in\mathcal{B}(Y,W)$. Then also the restriction $T|_{(X,Y)_{\theta,p}}\in\mathcal{B}\big((X,Y)_{\theta,p},(V,W)_{\theta,p}\big)$ is bounded, with operator norm
\begin{equation*}
\Vert T|_{(X,Y)_{\theta,p}}\Vert\leq\Vert T|_X\Vert^{1-\theta}\Vert T|_Y\Vert^\theta.
\end{equation*}
\end{theorem}

\begin{proof}
See \cite[Theorem 1.1.6]{Ale_INTER}.
\end{proof}

Beside the interpolation spaces \eqref{Eq_Interpolation_space}, which are embedded in between $X\cap Y$ and $X+Y$ by Lemma \ref{lem_Properties_interpolation_spaces} iv), we will now discuss classes of intermediate spaces in between $X\cap Y$ and $X+Y$.

\begin{definition}[Intermediate spaces]\label{defi_Intermediate_spaces}
Let $X,Y$ be an interpolation couple and $\theta\in(0,1)$. For right Banach spaces $E$ with $X\cap Y\hookrightarrow E\hookrightarrow X+Y$, we now consider two classes of \textit{intermediate spaces}:

\medskip

We say $E\in J_\theta(X,Y)$, if there exists $C_J\geq 0$, such that
\begin{equation*}
\Vert x\Vert_E\leq C_J\Vert x\Vert_X^{1-\theta}\Vert x\Vert_Y^\theta,\qquad x\in X\cap Y.
\end{equation*}
We say $E\in K_\theta(X,Y)$, if there exists some $C_K\geq 0$, such that
\begin{equation*}
K(t,x)\leq C_Kt^\theta\Vert x\Vert_E,\qquad t>0,\,x\in E.
\end{equation*}
\end{definition}

There exists the following characterization of the set of intermediate spaces $J_\theta(X,Y)$ and $K_\theta(X,Y)$, which connects them to the interpolation spaces \eqref{Eq_Interpolation_space}.

\begin{proposition}
Let $X,Y$ be an interpolation couple, $\theta\in(0,1)$ and $E$ a right Banach space with $X\cap Y\hookrightarrow E\hookrightarrow X+Y$. Then there holds

\begin{enumerate}
\item[i)] $E\in J_\theta(X,Y)\Leftrightarrow(X,Y)_{\theta,1}\hookrightarrow E$.
\item[ii)] $E\in K_\theta(X,Y)\Leftrightarrow E\hookrightarrow(X,Y)_{\theta,\infty}$.
\end{enumerate}
\end{proposition}

\begin{proof}
See \cite[Theorem 3.2.5]{BERG_INTER}.
\end{proof}

In particular, it follows from Lemma \ref{lem_Properties_interpolation_spaces} iv), that for every $\theta\in(0,1)$, $p\in[1,\infty]$ there holds
\begin{equation*}
(X,Y)_{\theta,p}\in J_\theta(X,Y)\cap K_\theta(X,Y).
\end{equation*}
The final example is taken from \cite[Example 1.3.3]{Ale_INTER} and shows that there indeed exist intermediate spaces which are not interpolation spaces.

\begin{example}
Let $C^k(\mathbb{R}^n)$ be the space of $k$-times continuously differentiable functions on $\mathbb{R}^n$ with bounded derivatives. Then
\begin{equation*}
C^1(\mathbb{R}^n)\in J_{\frac{1}{2}}\big(C(\mathbb{R}^n),C^2(\mathbb{R}^n)\big)\cap K_{\frac{1}{2}}\big(C(\mathbb{R}^n),C^2(\mathbb{R}^n)\big),
\end{equation*}
but $C^1(\mathbb{R}^n)$ is not of the form $\big(C(\mathbb{R}^n),C^2(\mathbb{R}^n)\big)_{\frac{1}{2},p}$ for any $p\in[1,\infty]$.
\end{example}

The last result of this section is the following reiteration theorem, which tells what happens if we interpolate in between intermediate spaces.

\begin{theorem}[Reiteration theorem]
Let $X,Y$ be an interpolation , $\theta_0,\theta_1,\theta\in(0,1)$, $p\in[1,\infty]$. Then for $E_0\in J_{\theta_0}(X,Y)\cap K_{\theta_0}(X,Y)$ and $E_1\in J_{\theta_1}(X,Y)\cap K_{\theta_1}(X,Y)$, there holds
\begin{equation*}
(E_0,E_1)_{\theta,p}=(X,Y)_{(1-\theta)\theta_0+\theta\theta_1,p},\quad\text{with equivalent norms}.
\end{equation*}
\end{theorem}

\begin{proof}
See \cite[Theorem 1.3.5]{Ale_INTER}.
\end{proof}

\begin{corollary}
Let $X,Y$ be an interpolation couple, $\theta_0,\theta_1,\theta\in(0,1)$, $p_0,p_1,p\in[1,\infty]$. Then there holds
\begin{equation*}
\big((X,Y)_{\theta_0,p_0},(X,Y)_{\theta_1,p_1}\big)_{\theta,p}=(X,Y)_{(1-\theta)\theta_0+\theta\theta_1,p},\quad\text{with equivalent norms.}
\end{equation*}
\end{corollary}

\section{Interpolating between domains of powers of operators}

In this section we consider the particular class of unbounded right linear operators $T$ in right Banach spaces $X$, given in Definition \ref{defi_Operators_of_minimal_growth}. First, Theorem \ref{thm_X_DTn_characterization} characterizes the interpolation spaces $(X,D(T^n))_{\theta,p}$ in terms of estimates of the resolvent operator \eqref{Eq_Q_operator}. Secondly, in Theorem \ref{thm_J_intermediate} and Theorem \ref{thm_K_intermediate} we prove that for $n<k<m\in \mathbb{N}_0$, the domain $D(T^k)$ is an intermediate space between $D(T^n)$ and $D(T^m)$ in the sense of Definition \ref{defi_Intermediate_spaces}.

\medskip

While all the results in Section \ref{sec_Interpolation_theory_of_quaternionic_Banach_spaces} are mainly copies of their analogs in complex interpolation theory, a significant difference arises in the definition of quaternionic spectra and resolvent operators. In the complex case the spectral problem is related to the bounded invertibility of the operator $\lambda-A$, for $\lambda\in\mathbb{C}$. For bounded operators $A\in\mathcal{B}(X_{\mathbb{C}})$, this inverse is for example given by the series
\begin{equation}\label{Eq_Resolvent_complex}
\sum\limits_{n=0}^\infty A^n\lambda^{-n-1}=(\lambda-A)^{-1},\qquad|\lambda|>\Vert A\Vert.
\end{equation}
However, for quaternionic operators $T\in\mathcal{B}(X)$ and quaternionic numbers $s\in\mathbb{H}$, the explicit form of this series is given by
\begin{equation}\label{Eq_Resolvent_quaternionic}
\sum\limits_{n=0}^\infty T^ns^{-n-1}=(T^2-2s_0T+|s|^2)^{-1}(\overline{s}-T),\qquad|s|>\Vert T\Vert,
\end{equation}
where the discrepancy between \eqref{Eq_Resolvent_complex} and \eqref{Eq_Resolvent_quaternionic} comes from the noncommutativity of quaternionic numbers and operators. Motivated by the explicit value of the series \eqref{Eq_Resolvent_quaternionic}, we now define the following notion of spectrum, for more details see \cite{CGK}.

\begin{definition}[$S$-spectrum]
Let $T:D(T)\subseteq X\rightarrow X$ be a closed right linear operator. For every $s\in\mathbb{H}$ we consider the operator
\begin{equation}\label{Eq_Q_operator}
Q_s(T):=T^2-2s_0T+|s|^2,\qquad\text{with }D(Q_s(T)):=D(T^2),
\end{equation}
and define the \textit{$S$-resolvent set} and the \textit{$S$-spectrum} as
\begin{equation*}
\rho_S(T):=\Set{s\in\mathbb{H} | Q_s^{-1}(T)\in\mathcal{B}(X)}\qquad\text{and}\qquad\sigma_S(T):=\mathbb{H}\setminus\rho_S(T).
\end{equation*}
\end{definition}

\begin{definition}\label{defi_Operators_of_minimal_growth}
Let $X$ be a right Banach space and $T:D(T)\subseteq X\rightarrow X$ a right linear, closed operator. We assume that there exists some $\omega\in[0,\pi]$, such that
\begin{equation}\label{Eq_Somega}
S_\omega:=\Set{te^{i\omega} | t>0,\,i\in\mathbb{S}}\subseteq\rho_S(T),
\end{equation}
and there exists some $M\geq 0$ such that the inverse of \eqref{Eq_Q_operator} is bounded by
\begin{equation}\label{Eq_Q_TQ_estimate}
\Vert Q_s^{-1}(T)\Vert\leq\frac{M}{|s|^2}\qquad\text{and}\qquad\Vert TQ_s^{-1}(T)\Vert\leq\frac{M}{|s|},\qquad s\in S_\omega.
\end{equation}
\end{definition}

The estimates \eqref{Eq_Q_TQ_estimate} can now be generalized also to other powers of $T$ and $Q_s^{-1}(T)$.

\begin{lemma}\label{lem_Q_resolvent_estimate}
Let $T$ be as in Definition \ref{defi_Operators_of_minimal_growth}. Then with the constant $M\geq 0$ from \eqref{Eq_Q_TQ_estimate} there holds for every $n,m\in\mathbb{N}_0$ with $n\leq 2m$ the estimate
\begin{equation}\label{Eq_Q_resolvent_estimate}
\Vert T^nQ_s^{-m}(T)\Vert\leq\frac{(1+3M)^m}{|s|^{2m-n}},\qquad s\in S_\omega.
\end{equation}
\end{lemma}

\begin{proof}
Let $s\in S_\omega$. We will repeatedly apply \eqref{Eq_Q_TQ_estimate}, but distinguish the cases:

\medskip

$\circ$\;\;In the case $0\leq n\leq m$, we get
\begin{equation*}
\Vert T^nQ_s^{-m}(T)\Vert=\Vert(TQ_s^{-1}(T))^nQ_s^{-(m-n)}(T)\Vert\leq\Big(\frac{M}{|s|}\Big)^n\Big(\frac{M}{|s|^2}\Big)^{m-n}=\frac{M^m}{|s|^{2m-n}}.
\end{equation*}
$\circ$\;\;In the case $m<n\leq 2m$ on the other hand, we get
\begin{align*}
\Vert T^nQ_s^{-m}(T)\Vert&=\big\Vert\big(T^2Q_s^{-1}(T)\big)^{n-m}\big(TQ_s^{-1}(T)\big)^{2m-n}\big\Vert \\
&=\big\Vert\big(1+2s_0TQ_s^{-1}(T)-|s|^2Q_s^{-1}(T)\big)^{n-m}\big(TQ_s^{-1}(T)\big)^{2m-n}\big\Vert \\
&\leq\Big(1+2|s_0|\frac{M}{|s|}+|s|^2\frac{M}{|s|^2}\Big)^{n-m}\Big(\frac{M}{|s|}\Big)^{2m-n}\leq\frac{(1+3M)^m}{|s|^{2m-n}}. \qedhere
\end{align*}
\end{proof}

Next we will prove that the powers $T^n$ are closed operators, i.e. $D(T^n)$ are Banach spaces, and that these domains are continuously embedded in one another. In particular this shows that any pair $D(T^n)$, $D(T^m)$ is an interpolation couple, and interpolating between these spaces makes sense.

\begin{lemma}\label{lem_DT_embedding}
Let $T$ be as in Definition \ref{defi_Operators_of_minimal_growth}. Then for every $n\in\mathbb{N}_0$ the power $T^n$ is a closed operator, and for every $n\leq m\in\mathbb{N}_0$ there holds
\begin{equation}\label{Eq_DT_embedding}
D(T^m)\hookrightarrow D(T^n).
\end{equation}
\end{lemma}

\begin{proof}
Since $\rho_S(T)\setminus\{0\}\neq\emptyset$, by the assumption \eqref{Eq_Somega}, let us choose some arbitrary $0\neq s\in\rho_S(T)$. For this $s$ we can write
\begin{equation}\label{Eq_DT_embedding_1}
T^n=T^nQ_s^n(T)Q_s^{-n}(T)=Q_s^n(T)(TQ_s^{-1}(T))^n.
\end{equation}
Now we observe that $TQ_s^{-1}(T)\in\mathcal{B}(X)$ is bounded, as the everywhere defined product of a closed and a bounded operator. Moreover, $Q_s^n(T)$ is closed, as it is the inverse of $Q_s^{-n}(T)\in\mathcal{B}(X)$. Consequently, the power $T^n$ in \eqref{Eq_DT_embedding_1} is the product of a closed and a bounded operator, and hence is closed.

\medskip

For the proof of the embedding \eqref{Eq_DT_embedding}, we use a similar formula as in \eqref{Eq_DT_embedding_1} to decompose every $x\in D(T^m)$ into
\begin{equation*}
x=(T^2+2s_0T+|s|^2)^mQ_s^{-m}(T)x=\sum\limits_{\stackrel{\alpha,\beta,\gamma=0}{\mathsmaller{\alpha+\beta+\gamma=m}}}^m\frac{m!}{\alpha!\beta!\gamma!}T^{2\alpha}(2s_0T)^\beta|s|^{2\gamma}Q_s^{-m}(T)x.
\end{equation*}
In this representation, we can use \eqref{Eq_Q_resolvent_estimate}, to estimate
\begin{align*}
\Vert T^nx\Vert&\leq\sum\limits_{\stackrel{\alpha,\beta,\gamma=0}{\mathsmaller{\alpha+\beta+\gamma=m,2\alpha+\beta<m-n}}}^m\frac{m!2^\beta|s|^{\beta+2\gamma}}{\alpha!\beta!\gamma!}\Vert T^{n+2\alpha+\beta}Q_s^{-m}(T)x\Vert \\
&\quad+\sum\limits_{\stackrel{\alpha,\beta,\gamma=0}{\mathsmaller{\alpha+\beta+\gamma=m,2\alpha+\beta\geq m-n}}}^m\frac{m!2^\beta|s|^{\beta+2\gamma}}{\alpha!\beta!\gamma!}\Vert T^{n+2\alpha+\beta-m}Q_s^{-m}(T)T^mx\Vert \\
&\leq\sum\limits_{\stackrel{\alpha,\beta,\gamma=0}{\mathsmaller{\alpha+\beta+\gamma=m}}}^m\frac{m!2^\beta|s|^n}{\alpha!\beta!\gamma!}(1+3M)^m\Vert x\Vert+\sum\limits_{\stackrel{\alpha,\beta,\gamma=0}{\mathsmaller{\alpha+\beta+\gamma=m}}}^m\frac{m!2^\beta}{\alpha!\beta!\gamma!}\frac{(1+3M)^m}{|s|^{m-n}}\Vert T^mx\Vert \\
&=(4+12M)^m|s|^n\Vert x\Vert+\frac{(4+12M)^m}{|s|^{m-n}}\Vert T^mx\Vert.
\end{align*}
This estimate then shows the embedding
\begin{equation*}
\Vert x\Vert_{D(T^n)}\leq\max\Big\{1+(4+12M)^m|s|^n,\frac{(4+12M)^m}{|s|^{m-n}}\Big\}\Vert x\Vert_{D(T^m)}. \qedhere
\end{equation*}
\end{proof}

The following theorem characterizes the interpolation spaces between $X$ and the domains $D(T^n)$, in terms of the pseudo $S$-resolvent operator, that is the inverse of the operator $Q_s(T)$ defined in \eqref{Eq_Q_operator}. This result is the quaternionic analog, but also a generalization, of the complex results in \cite[Proposition 3.1.1 \& Proposition 3.1.6]{Ale_INTER}. The major difference between the complex and the quaternionic case is that the function $\psi_x$ in \eqref{Eq_psix} has an additional term for odd values $n$. This is mainly due to the second order resolvent $Q_s^{-1}(T)$.

\begin{theorem}\label{thm_X_DTn_characterization}
Let $T$ be as in Definition \ref{defi_Operators_of_minimal_growth} and $n\in\mathbb{N}$. For every $x\in X$ let
\begin{equation}\label{Eq_psix}
\psi_x(t):=\begin{cases} \Vert T^nQ_{te^{i\omega}}^{-\frac{n}{2}}(T)x\Vert, & \text{n } \text{ even,} \\ \Vert T^{n+1}Q_{te^{i\omega}}^{-\frac{n+1}{2}}(T)x\Vert+t\Vert T^nQ_{te^{i\omega}}^{-\frac{n+1}{2}}(T)x\Vert, & \text{n } \text{odd,} \end{cases}\qquad t>0.
\end{equation}
Then, for every $\theta\in(0,1)$ and $p\in[1,\infty]$ we obtain the representation
\begin{equation*}
(X,D(T^n))_{\theta,p}=\Set{x\in X | t\mapsto t^{n\theta}\psi_x(t)\in L^p_*(0,\infty)}
\end{equation*}
of the interpolation space, and the norms
\begin{equation*}
\Vert x\Vert_{\theta,p}\quad\text{and}\quad\Vert x\Vert_{\theta,p}^*:=\Vert x\Vert+\Vert t^{n\theta}\psi_x(t)\Vert_{L^p_*(0,\infty)}\quad\text{are equivalent.}
\end{equation*}
\end{theorem}

Note that the operator $Q_{te^{i\omega}}^{-1}(T)=(T^2+2tT\cos\omega+t^2)^{-1}$ does not depend on the imaginary unit $i\in\mathbb{S}$.

\begin{proof}
For the embedding $\text{\grqq}\hookrightarrow\text{\grqq}$ let $x\in(X,D(T^n))_{\theta,p}$. For any $t>0$ and any decomposition $x=a+b$ with $a\in X$, $b\in D(T^n)$, we want to estimate $\psi_x(t)$.

\medskip

$\circ$\;\;If $n$ is even, we use \eqref{Eq_Q_resolvent_estimate}, to estimate
\begin{equation}\label{Eq_X_DT_interpolation_1}
\psi_x(t)\leq\Vert T^nQ_{te^{i\omega}}^{-\frac{n}{2}}(T)a\Vert+\Vert Q_{te^{i\omega}}^{-\frac{n}{2}}(T)T^nb\Vert\leq(1+3M)^{\frac{n}{2}}\Big(\Vert a\Vert+\frac{\Vert T^nb\Vert}{t^n}\Big).
\end{equation}
$\circ$\;\;If $n$ is odd, we similarly get
\begin{align}
\psi_x(T)&=\Vert T^{n+1}Q_{te^{i\omega}}^{-\frac{n+1}{2}}(T)x\Vert+t\Vert T^nQ_{te^{i\omega}}^{-\frac{n+1}{2}}(T)x\Vert \notag \\
&\leq(1+3M)^{\frac{n+1}{2}}\Big(\Vert a\Vert+\frac{\Vert T^nb\Vert}{t^n}+t\frac{\Vert a\Vert}{t}+t\frac{\Vert T^nb\Vert}{t^{n+1}}\Big) \notag \\
&=2(1+3M)^{\frac{n+1}{2}}\Big(\Vert a\Vert+\frac{\Vert T^nb\Vert}{t^n}\Big). \label{Eq_X_DT_interpolation_2}
\end{align}
Since the estimates \eqref{Eq_X_DT_interpolation_1} and \eqref{Eq_X_DT_interpolation_2} hold true for every decomposition $x=a+b$, we can take the respective infimum of the right hand side. Hence, for every $n\in\mathbb{N}$ we end up with an estimate of the $K$-functional \eqref{Eq_K_functional}, namely
\begin{equation*}
\psi_x(t)\leq 2(1+M)^{\lceil\frac{n}{2}\rceil}K\Big(\frac{1}{t^n},x\Big),\qquad t>0,
\end{equation*}
where we used the ceil function $\lceil\frac{n}{2}\rceil$, which gives the smallest integer larger or equal to $\frac{n}{2}$. If we multiply this inequality by $t^{n\theta}$ and take the $\Vert\cdot\Vert_{L^p_*(0,\infty)}$-norm, we get
\begin{align}
\Vert t^{n\theta}\psi_x(t)\Vert_{L^p_*(0,\infty)}&\leq 2(1+M)^{\lceil\frac{n}{2}\rceil}\Big\Vert t^{n\theta}K\Big(\frac{1}{t^n},x\Big)\Big\Vert_{L^p_*(0,\infty)} \notag \\
&=\frac{2(1+M)^{\lceil\frac{n}{2}\rceil}}{n^{\frac{1}{p}}}\Vert t^{-\theta}K(t,x)\Vert_{L^p_*(0,\infty)}=\frac{2(1+M)^{\lceil\frac{n}{2}\rceil}}{n^{\frac{1}{p}}}\Vert x\Vert_{\theta,p}, \label{Eq_X_DT_interpolation_3}
\end{align}
where in the second line we used \eqref{Eq_Lpstar_symmetry} with the convention $n^{\frac{1}{\infty}}:=1$. Moreover, by Lemma \ref{lem_Properties_interpolation_spaces} iv) and Lemma \ref{lem_DT_embedding} we have
\begin{equation*}
(X,D(T^n))_{\theta,p}\hookrightarrow X+D(T^n)\hookrightarrow X+X=X.
\end{equation*}
Hence $\Vert x\Vert\leq C\Vert x\Vert_{\theta,p}$ for some $C\geq 0$, and together with \eqref{Eq_X_DT_interpolation_3} we get
\begin{equation*}
\Vert x\Vert_{\theta,p}^*=\Vert x\Vert+\Vert t^{n\theta}\psi_x(t)\Vert_{L^p_*(0,\infty)}\leq\Big(C+\frac{2(1+M)^{\lceil\frac{n}{2}\rceil}}{n^{\frac{1}{p}}}\Big)\Vert x\Vert_{\theta,p}.
\end{equation*}
For the embedding $\text{\grqq}\hookleftarrow\text{\grqq}$ let $x\in X$ with $\psi_x\in L^p_*(0,\infty)$. For $0<t\leq 1$, we choose $a=x$ and $b=0$ in the definition \eqref{Eq_K_functional} of the $K$-functional, and get
\begin{equation}\label{Eq_X_DT_interpolation_9}
K\Big(\frac{1}{t^n},x\Big)\leq\Vert a\Vert+\frac{1}{t^n}\Vert b\Vert_{D(T^n)}=\Vert x\Vert,\qquad 0<t\leq 1.
\end{equation}
For $t\geq 1$ on the other hand, we distinguish two cases.

\medskip

$\circ$\;\;If $n$ is even, we decompose the vector $x$ into
\begin{align}
&x=(T^2+2tT\cos\omega+t^2)^nQ_{te^{i\omega}}^{-n}(T)x=\sum\limits_{\stackrel{\alpha,\beta,\gamma=0}{\mathsmaller{\alpha+\beta+\gamma=n}}}^n\frac{n!t^{2\gamma}}{\alpha!\beta!\gamma!}T^{2\alpha}(2tT\cos\omega)^\beta Q_{te^{i\omega}}^{-n}(T)x \notag \\
&=\underbrace{\sum\limits_{\stackrel{\alpha,\beta,\gamma=0}{\mathsmaller{\stackrel{\alpha+\beta+\gamma=n}{2\alpha+\beta\geq n+1}}}}^n\frac{n!(2\cos\omega)^\beta}{\alpha!\beta!\gamma!t^{-\beta-2\gamma}}T^{2\alpha+\beta}Q_{te^{i\omega}}^{-n}(T)x}_{=:a\in X}+\underbrace{\sum\limits_{\stackrel{\alpha,\beta,\gamma=0}{\mathsmaller{\stackrel{\alpha+\beta+\gamma=n}{2\alpha+\beta\leq n}}}}^n\frac{n!(2\cos\omega)^\beta}{\alpha!\beta!\gamma!t^{-\beta-2\gamma}}T^{2\alpha+\beta}Q_{te^{i\omega}}^{-n}(T)x.}_{=:b\in D(T^n)} \label{Eq_X_DT_interpolation_4}
\end{align}
With this decomposition and the inequality \eqref{Eq_Q_resolvent_estimate}, we can estimate the $K$-functional
\begin{align}
K\Big(\frac{1}{t^n},x\Big)&\leq\sum\limits_{\stackrel{\alpha,\beta,\gamma=0}{\mathsmaller{\stackrel{\alpha+\beta+\gamma=n}{2\alpha+\beta\geq n+1}}}}^n\frac{n!2^\beta}{\alpha!\beta!\gamma!}\frac{\Vert T^{2\alpha+\beta-n}Q_{te^{i\omega}}^{-\frac{n}{2}}(T)T^nQ_{te^{i\omega}}^{-\frac{n}{2}}(T)x\Vert}{t^{-\beta-2\gamma}} \notag \\
&+\sum\limits_{\stackrel{\alpha,\beta,\gamma=0}{\mathsmaller{\stackrel{\alpha+\beta+\gamma=n}{2\alpha+\beta\leq n}}}}^n\frac{n!2^\beta}{\alpha!\beta!\gamma!}\frac{\Vert T^{2\alpha+\beta}Q_{te^{i\omega}}^{-\frac{n}{2}}(T)T^nQ_{te^{i\omega}}^{-\frac{n}{2}}(T)x\Vert+\Vert T^{2\alpha+\beta}Q_{te^{i\omega}}^{-n}(T)x\Vert}{t^{n-\beta-\gamma}} \notag \\
&\leq\sum\limits_{\stackrel{\alpha,\beta,\gamma=0}{\mathsmaller{\stackrel{\alpha+\beta+\gamma=n}{2\alpha+\beta\geq n+1}}}}^n\frac{n!2^\beta(1+3M)^{\frac{n}{2}}}{\alpha!\beta!\gamma!}\Vert T^nQ_{te^{i\omega}}^{-\frac{n}{2}}(T)x\Vert \notag \\
&\quad+\sum\limits_{\stackrel{\alpha,\beta,\gamma=0}{\mathsmaller{\stackrel{\alpha+\beta+\gamma=n}{2\alpha+\beta\leq n}}}}^n\frac{n!2^\beta}{\alpha!\beta!\gamma!}\Big((1+3M)^{\frac{n}{2}}\Vert T^nQ_{te^{i\omega}}^{-\frac{n}{2}}(T)x\Vert+(1+3M)^n\frac{\Vert x\Vert}{t^n}\Big) \notag \\
&\leq\sum\limits_{\stackrel{\alpha,\beta,\gamma=0}{\mathsmaller{\alpha+\beta+\gamma=n}}}^n\frac{n!2^\beta}{\alpha!\beta!\gamma!}(1+3M)^{\frac{n}{2}}\psi_x(t)+\sum\limits_{\stackrel{\alpha,\beta,\gamma=0}{\mathsmaller{\alpha+\beta+\gamma=n}}}^n\frac{n!2^\beta}{\alpha!\beta!\gamma!}(1+3M)^n\frac{\Vert x\Vert}{t^n} \notag \\
&\leq 4^n(1+3M)^{\frac{n}{2}}\psi_x(t)+4^n(1+3M)^n\frac{\Vert x\Vert}{t^n}, \label{Eq_X_DT_interpolation_6}
\end{align}
where in the second last line we inserted the definition \eqref{Eq_psix} of $\psi_x(t)$ and in the last inequality we solved the trinomial sum $\sum\limits_{\stackrel{\alpha,\beta,\gamma=0}{\mathsmaller{\alpha+\beta+\gamma=n}}}^n\frac{n!2^\beta}{\alpha!\beta!\gamma!}=4^n$.

\medskip

$\circ$\;\;If $n$ is odd we again use the same decomposition $x=a+b$ as in \eqref{Eq_X_DT_interpolation_4}, and similarly estimate the $K$-functional by
\begin{align}
K\Big(\frac{1}{t^n},x\Big)&\leq\sum\limits_{\stackrel{\alpha,\beta,\gamma=0}{\mathsmaller{\stackrel{\alpha+\beta+\gamma=n}{2\alpha+\beta\geq n+1}}}}^n\frac{n!2^\beta}{\alpha!\beta!\gamma!}\frac{\Vert T^{2\alpha+\beta-n-1}Q^{-\frac{n-1}{2}}(T)T^{n+1}Q^{-\frac{n+1}{2}}(T)x\Vert}{t^{-\beta-2\gamma}} \notag \\
&+\sum\limits_{\stackrel{\alpha,\beta,\gamma=0}{\mathsmaller{\stackrel{\alpha+\beta+\gamma=n}{2\alpha+\beta\leq n}}}}^n\frac{n!2^\beta}{\alpha!\beta!\gamma!}\frac{\Vert T^{2\alpha+\beta}Q^{-\frac{n-1}{2}}(T)T^nQ^{-\frac{n+1}{2}}(T)x\Vert+\Vert T^{2\alpha+\beta}Q^{-n}(T)x\Vert}{t^{n-\beta-2\gamma}} \notag \\
&\leq\sum\limits_{\stackrel{\alpha,\beta,\gamma=0}{\mathsmaller{\stackrel{\alpha+\beta+\gamma=n}{2\alpha+\beta\geq n+1}}}}^n\frac{n!2^\beta(1+3M)^{\frac{n-1}{2}}}{\alpha!\beta!\gamma!}\Vert T^{n+1}Q^{-\frac{n+1}{2}}(T)x\Vert \notag \\
&\quad+\sum\limits_{\stackrel{\alpha,\beta,\gamma=0}{\mathsmaller{\stackrel{\alpha+\beta+\gamma=n}{2\alpha+\beta\leq n}}}}^n\frac{n!2^\beta}{\alpha!\beta!\gamma!}\Big((1+3M)^{\frac{n-1}{2}}t\Vert T^nQ^{-\frac{n+1}{2}}(T)x\Vert+(1+3M)^n\frac{\Vert x\Vert}{t^n}\Big) \notag \\
&\leq\sum\limits_{\stackrel{\alpha,\beta,\gamma=0}{\mathsmaller{\alpha+\beta+\gamma=n}}}^n\frac{n!2^\beta}{\alpha!\beta!\gamma!}(1+3M)^{\frac{n-1}{2}}\psi_x(t)+\sum\limits_{\stackrel{\alpha,\beta,\gamma=0}{\mathsmaller{\alpha+\beta+\gamma=n}}}^n\frac{n!2^\beta}{\alpha!\beta!\gamma!}(1+3M)^n\frac{\Vert x\Vert}{t^n} \notag \\
&\leq 4^n(1+3M)^{\frac{n-1}{2}}\psi_x(t)+(1+3M)^n\frac{\Vert x\Vert}{t^n}. \label{Eq_X_DT_interpolation_7}
\end{align}
The estimates \eqref{Eq_X_DT_interpolation_6} and \eqref{Eq_X_DT_interpolation_7} now show that the $K$-functional admits for some $M_1,M_2\geq 0$ and every $n\in\mathbb{N}$ the upper bound
\begin{equation}\label{Eq_X_DT_interpolation_8}
K\Big(\frac{1}{t^n},x\Big)\leq M_1\psi_x(t)+M_2\frac{\Vert x\Vert}{t^n},\qquad t\geq 1.
\end{equation}
The inequalities \eqref{Eq_X_DT_interpolation_9} and \eqref{Eq_X_DT_interpolation_8}, together with the identity \eqref{Eq_Lpstar_symmetry}, now give
\begin{align*}
\Vert x\Vert_{\theta,p}&=\Vert t^{-\theta}K(t,x)\Vert_{L^p_*(0,\infty)}=n^{\frac{1}{p}}\Big\Vert t^{n\theta}K\Big(\frac{1}{t^n},x\Big)\Big\Vert_{L^p_*(0,\infty)} \\
&\leq n^{\frac{1}{p}}\Big\Vert t^{n\theta}K\Big(\frac{1}{t^n},x\Big)\Big\Vert_{L^p_*(0,1)}+n^{\frac{1}{p}}\Big\Vert t^{n\theta}K\Big(\frac{1}{t^n},x\Big)\Big\Vert_{L^p_*(1,\infty)} \\
&\leq n^{\frac{1}{p}}\Vert x\Vert\Vert t^{n\theta}\Vert_{L^p_*(0,1)}+M_1n^{\frac{1}{p}}\Vert t^{n\theta}\psi_x(t)\Vert_{L^p_*(1,\infty)}+M_2n^{\frac{1}{p}}\Vert x\Vert\Vert t^{n(\theta-1)}\Vert_{L^p_*(1,\infty)} \\
&=\frac{\Vert x\Vert}{p^{\frac{1}{p}}}\Big(\frac{1}{\theta^{\frac{1}{p}}}+\frac{M_2}{(1-\theta)^{\frac{1}{p}}}\Big)+M_1n^{\frac{1}{p}}\Vert t^{n\theta}\psi_x(t)\Vert_{L^p_*(1,\infty)} \\
&\leq\max\bigg\{\frac{1}{p^{\frac{1}{p}}}\Big(\frac{1}{\theta^{\frac{1}{p}}}+\frac{M_2}{(1-\theta)^{\frac{1}{p}}}\Big),M_1n^{\frac{1}{p}}\bigg\}\Vert x\Vert_{\theta,p}^*. \qedhere
\end{align*}
\end{proof}

The upcoming two Theorems \ref{thm_J_intermediate} \& \ref{thm_K_intermediate} show that the domain $D(T^k)$ is an intermediate space between $D(T^n)$ and $D(T^m)$ according to Definition \ref{defi_Intermediate_spaces}.

\begin{theorem}\label{thm_J_intermediate}
Let $T$ be as in Definition \ref{defi_Operators_of_minimal_growth}. Then for every $n<k<m\in\mathbb{N}_0$, there holds
\begin{equation*}
D(T^k)\in J_{\frac{k-n}{m-n}}(D(T^n),D(T^m));
\end{equation*}
i.e. according to Definition \ref{defi_Intermediate_spaces} there exists some $C_{n,k,m}\geq 0$, such that
\begin{equation*}
\Vert x\Vert_{D(T^k)}\leq C_{n,k,m}\Vert x\Vert_{D(T^n)}^{\frac{m-k}{m-n}}\Vert x\Vert_{D(T^m)}^{\frac{k-n}{m-n}},\qquad x\in D(T^m).
\end{equation*}
\end{theorem}

\begin{proof}
Note that it is enough to prove that for some $C_{n,k,m}\geq 0$, there holds
\begin{equation}\label{Eq_J_1}
\Vert T^kx\Vert\leq C_{n,k,m}\Vert T^nx\Vert^{\frac{m-k}{m-n}}\Vert T^mx\Vert^{\frac{k-n}{m-n}},\qquad x\in D(T^m).
\end{equation}
Indeed, if \eqref{Eq_J_1} holds true, we get
\begin{align*}
\Vert x\Vert_{D(T^k)}&=\Vert x\Vert+\Vert T^kx\Vert\leq\Vert x\Vert+C_{n,k,m}\Vert T^nx\Vert^{\frac{m-k}{m-n}}\Vert T^mx\Vert^{\frac{k-n}{m-n}} \\
&\leq\max\{1,C_{n,k,m}\}\big(\Vert x\Vert+\Vert T^nx\Vert^{\frac{m-k}{m-n}}\Vert T^mx\Vert^{\frac{k-n}{m-n}}\big) \\
&\leq\max\{1,C_{n,k,m}\}\big(\Vert x\Vert+\Vert T^nx\Vert\big)^{\frac{m-k}{m-n}}\big(\Vert x\Vert+\Vert T^mx\Vert\big)^{\frac{k-n}{m-n}} \\
&=\max\{1,C_{n,k,m}\}\Vert x\Vert_{D(T^n)}^{\frac{m-k}{m-n}}\Vert x\Vert_{D(T^m)}^{\frac{k-n}{m-n}},
\end{align*}
where in the third line we used the basic inequality
\begin{equation*}
c+a^{1-\theta}b^\theta\leq(c+a)^{1-\theta}(c+b)^\theta,\qquad a,b,c\geq 0,\,\ \ \theta\in(0,1).
\end{equation*}
Note, that this inequality is trivial for $c=0$, but also for $c>0$ it is true since the $c$-derivative of the left side is always smaller than the one of the right side,
\begin{equation*}
\frac{d}{dc}(c+a^{1-\theta}b^\theta)=1\leq (1-\theta)\Big(\frac{c+b}{c+a}\Big)^\theta+\theta\Big(\frac{c+a}{c+b}\Big)^{1-\theta}=\frac{d}{dc}(c+a)^{1-\theta}(c+b)^\theta,
\end{equation*}
where the inequality comes from the convexity of the function $[0,1]\ni x\mapsto(\frac{c+b}{c+a})^x$.

\medskip

Hence it is left to prove is the inequality \eqref{Eq_J_1}, which will be done for fixed $n<k\in\mathbb{N}_0$ and by induction with respect to $m\in\{k+1,\dots\}$. For the induction start $m=k+1$, let $x\in D(T^{k+1})$ and consider the function
\begin{equation*}
f_x(t):=t^{2(k-n)}Q_{te^{i\omega}}^{-(k-n)}(T)x,\qquad t>0.
\end{equation*}
Thanks to Lemma \ref{lem_Q_resolvent_estimate} we compute the limit
\begin{align*}
\lim\limits_{t\rightarrow\infty}\Vert t^2Q_{te^{i\omega}}^{-1}(T)y-y\Vert&=\lim\limits_{t\rightarrow\infty}\Vert(T-2t\cos\omega)Q_{te^{i\omega}}^{-1}(T)Ty\Vert \\
&\leq\lim\limits_{t\rightarrow\infty}\frac{(1+3M)(1+2\cos\omega)}{t}\Vert Ty\Vert=0,\qquad y\in D(T),
\end{align*}
so there is also the convergence
\begin{equation}\label{Eq_J_2}
\lim\limits_{t\rightarrow\infty}f_x(t)=\lim\limits_{t\rightarrow\infty}(t^2Q_{te^{i\omega}}^{-1}(T))^{k-n}x=x.
\end{equation}
Secondly, the derivative of the function is given by
\begin{equation*}
f_x'(t)=2(k-n)t^{2(k-n)-1}(T^2-tT\cos\omega)Q_{te^{i\omega}}^{-(k-n+1)}(T)x.
\end{equation*}
For some arbitrary $\lambda>0$ we can now integrate this derivative over the interval $[\lambda,\infty)$. Using also the limit \eqref{Eq_J_2}, this gives
\begin{equation*}
\int_\lambda^\infty f_x'(t)dt=\lim\limits_{t\rightarrow\infty}f_x(t)-f_x(\lambda)=x-\lambda^{2(k-n)}Q_{\lambda e^{i\omega}}^{-(k-n)}(T)x.
\end{equation*}
Bringing the term $\lambda^{2(k-n)}Q_{\lambda e^{i\omega}}^{-(k-n)}(T)x$ to the other side and applying the operator $T^k$ on both sides of this equation, allows us to estimate
\begin{align}
\Vert T^kx\Vert&=\lambda^{2(k-n)}\Vert T^kQ_{\lambda e^{i\omega}}^{-(k-n)}(T)x\Vert+\int_\lambda^\infty\Vert T^kf_x'(t)\Vert dt \notag \\
&=\lambda^{2(k-n)}\Vert T^{k-n}Q_{\lambda e^{i\omega}}^{-(k-n)}(T)T^nx\Vert \notag \\
&\quad+2(k-n)\int_\lambda^\infty t^{2(k-n)-1}\Vert(T-t\cos\omega)Q_{te^{i\omega}}^{-(k-n+1)}T^{k+1}x\Vert dt \notag \\
&\leq\lambda^{k-n}(1+3M)^{k-n}\Vert T^nx\Vert+4(k-n)(1+3M)^{k-n+1}\Vert T^{k+1}x\Vert\int_\lambda^\infty\frac{1}{t^2}dt \notag \\
&=\lambda^{k-n}(1+3M)^{k-n}\Vert T^nx\Vert+\frac{4(k-n)(1+3M)^{k-n+1}\Vert T^{k+1}x\Vert}{\lambda}, \label{Eq_J_4}
\end{align}
where in the third line we used Lemma \ref{lem_Q_resolvent_estimate}. Next we will minimize the right hand side with respect to the parameter $\lambda>0$.

\medskip

$\circ$\;\;First we consider the case $T^{k+1}x\neq 0$. Since $k+1\geq n$ also $T^nx\neq 0$, and the minimum of the right hand side is attained at the point $\lambda=\big(\frac{4(1+3M)\Vert T^{k+1}x\Vert}{\Vert T^nx\Vert}\big)^{\frac{1}{k-n+1}}>0$. Plugging this value $\lambda$ into \eqref{Eq_J_4}, gives the upper bound
\begin{equation}\label{Eq_J_3}
\Vert T^kx\Vert\leq(k-n+1)4^{\frac{k-n}{k-n+1}}(1+3M)^{\frac{(k-n)(k-n+2)}{k-n+1}}\Vert T^nx\Vert^{\frac{1}{k-n+1}}\Vert T^{k+1}x\Vert^{\frac{k-n}{k-n+1}}.
\end{equation}
$\circ$\;\;If $T^{k+1}x=0$, the estimate \eqref{Eq_J_4} looks like
\begin{equation*}
\Vert T^kx\Vert\leq\lambda^{k-n}(1+3M)^{k-n}\Vert T^nx\Vert.
\end{equation*}
But since $\lambda>0$ is arbitrary, we conclude $\Vert T^kx\Vert=0$ and the estimate \eqref{Eq_J_3} is satisfied also in this case. Hence in both cases the estimate \eqref{Eq_J_3} is satisfied, which is exactly the stated inequality \eqref{Eq_J_1} with $m=k+1$.

\medskip

For the induction step $m\rightarrow m+1$ let $x\in D(T^{m+1})$. Using the induction assumption once for $n<k<m$ in the first upcoming inequality, and with $n<m<m+1$ in the second inequality, gives
\begin{align*}
\Vert T^kx\Vert&\leq C_{n,k,m}\Vert T^nx\Vert^{\frac{m-k}{m-n}}\Vert T^mx\Vert^{\frac{k-n}{m-n}} \\
&\leq C_{n,k,m}\Vert T^nx\Vert^{\frac{m-k}{m-n}}\big(C_{n,m,m+1}\Vert T^nx\Vert^{\frac{1}{m-n+1}}\Vert T^{m+1}x\Vert^{\frac{m-n}{m-n+1}}\big)^{\frac{k-n}{m-n}} \\
&=C_{n,k,m}C_{n,m,m+1}^{\frac{k-n}{m-n}}\Vert T^nx\Vert^{\frac{m-k+1}{m-n+1}}\Vert T^{m+1}x\Vert^{\frac{k-n}{m-n+1}}. \qedhere
\end{align*}
\end{proof}

\begin{theorem}\label{thm_K_intermediate}
Let $T$ be as in Definition \ref{defi_Operators_of_minimal_growth}. Then for every $n<k<m\in\mathbb{N}_0$, there holds
\begin{equation*}
D(T^k)\in K_{\frac{k-n}{m-n}}(D(T^n),D(T^m));
\end{equation*}
i.e. according to Definition \ref{defi_Intermediate_spaces} there exists some $C_{n,k,m}\geq 0$, such that
\begin{equation}\label{Eq_K_intermediate}
K(t,x)\leq C_{n,k,m}t^{\frac{k-n}{m-n}}\Vert x\Vert_{D(T^k)},\qquad t>0,\,x\in D(T^k),
\end{equation}
where the $K$-functional is the one of the couple $D(T^n),D(T^m)$.
\end{theorem}

\begin{proof}
Let $x\in D(T^k)$. For every $t\geq 1$, we can choose $a=x$ and $b=0$ to estimate the $K$-functional by
\begin{equation}\label{Eq_K_intermediate_3}
K(t,x)\leq\Vert a\Vert_{D(T^n)}+t\Vert b\Vert_{D(T^m)}=\Vert x\Vert_{D(T^n)}\leq C_{n,m}t^{\frac{k-n}{m-n}}\Vert x\Vert_{D(T^k)},\qquad t\geq 1,
\end{equation}
where in the last inequality we used $t\geq 1$ as well as the constant $C_{n,m}$ from the embedding $D(T^m)\hookrightarrow D(T^n)$ in Lemma \ref{lem_DT_embedding}. For $0<t\leq 1$ on the other hand we choose $\tau=t^{-\frac{1}{m-n}}\geq 1$, and consider a similar splitting as in \eqref{Eq_X_DT_interpolation_4}, namely
\begin{align*}
x=&\underbrace{\sum\limits_{\stackrel{\alpha,\beta,\gamma=0}{\mathsmaller{\stackrel{\alpha+\beta+\gamma=m}{2\alpha+\beta\geq k}}}}^m\frac{m!(2\cos\omega)^\beta}{\alpha!\beta!\gamma!\tau^{-\beta-2\gamma}}T^{2\alpha+\beta}Q_{\tau e^{i\omega}}^{-m}(T)x}_{=:a\in D(T^k)\subseteq D(T^n)} \\
&\hspace{1cm}+\underbrace{\sum\limits_{\stackrel{\alpha,\beta,\gamma=0}{\mathsmaller{\stackrel{\alpha+\beta+\gamma=m}{2\alpha+\beta\leq k-1}}}}^m\frac{m!(2\cos\omega)^\beta}{\alpha!\beta!\gamma!\tau^{-\beta-2\gamma}}T^{2\alpha+\beta}Q_{\tau e^{i\omega}}^{-m}(T)x}_{=:b\in D(T^{2m+1})\subseteq D(T^m)}.
\end{align*}
Using Lemma \ref{lem_Q_resolvent_estimate}, the norms of these vectors $a$ and $b$ can now be estimated by
\begin{align}
\Vert a\Vert_{D(T^n)}&\leq\sum\limits_{\stackrel{\alpha,\beta,\gamma=0}{\mathsmaller{\stackrel{\alpha+\beta+\gamma=m}{2\alpha+\beta\geq k}}}}^m\frac{m!2^\beta}{\alpha!\beta!\gamma!}\frac{\Vert T^{2\alpha+\beta-k}Q_{\tau e^{i\omega}}^{-m}(T)T^kx\Vert+\Vert T^{n-k+2\alpha+\beta}Q_{\tau e^{i\omega}}^{-m}(T)T^kx\Vert}{\tau^{-\beta-2\gamma}} \notag \\
&\leq\sum\limits_{\stackrel{\alpha,\beta,\gamma=0}{\mathsmaller{\stackrel{\alpha+\beta+\gamma=m}{2\alpha+\beta\geq k}}}}^m\frac{m!2^\beta(1+3M)^m}{\alpha!\beta!\gamma!}\Big(\frac{1}{\tau^k}+\frac{1}{\tau^{k-n}}\Big)\Vert T^kx\Vert \notag \\
&\leq4^m(1+3M)^m\Big(\frac{1}{\tau^k}+\frac{1}{\tau^{k-n}}\Big)\Vert T^kx\Vert\leq 2(4+12M)^m\frac{\Vert T^kx\Vert}{\tau^{k-n}}, \label{Eq_K_intermediate_1}
\end{align}
as well as
\begin{align}
\Vert b\Vert_{D(T^m)}&\leq\sum\limits_{\stackrel{\alpha,\beta,\gamma=0}{\mathsmaller{\stackrel{\alpha+\beta+\gamma=m}{2\alpha+\beta\leq k-1}}}}^m\frac{m!2^\beta}{\alpha!\beta!\gamma!}\frac{\Vert T^{2\alpha+\beta}Q_{\tau e^{i\omega}}^{-m}(T)x\Vert+\Vert T^{m-k+2\alpha+\beta}Q_{\tau e^{i\omega}}^{-m}(T)T^kx\Vert}{\tau^{-\beta-2\gamma}} \notag \\
&\leq\sum\limits_{\stackrel{\alpha,\beta,\gamma=0}{\mathsmaller{\stackrel{\alpha+\beta+\gamma=m}{2\alpha+\beta\leq k-1}}}}^m\frac{m!2^\beta(1+3M)^m}{\alpha!\beta!\gamma!}\big(\Vert x\Vert+\tau^{m-k}\Vert T^kx\Vert\big) \notag \\
&\leq 4^m(1+3M)^m\big(\Vert x\Vert+\tau^{m-k}\Vert T^kx\Vert\big), \label{Eq_K_intermediate_2}
\end{align}
where from the respective second to third lines we added the missing terms in the sum by removing the conditions $2\alpha+\beta\geq k$ and $2\alpha+\beta\leq k-1$, and used the trinomial sum $\sum\limits_{\alpha,\beta,\gamma=0,\alpha+\beta+\gamma=m}^m\frac{m!2^\beta}{\alpha!\beta!\gamma!}=4^m$. With \eqref{Eq_K_intermediate_1} and \eqref{Eq_K_intermediate_2} we then get
\begin{align}
K(t,x)&\leq\Vert a\Vert_{D(T^n)}+t\Vert b\Vert_{D(T^m)} \notag \\
&\leq(4+12M)^m\Big(\frac{2\Vert T^kx\Vert}{\tau^{k-n}}+t\Vert x\Vert+t\tau^{m-k}\Vert T^kx\Vert\Big) \notag \\
&=(4+12M)^m\Big(3t^{\frac{k-n}{m-n}}\Vert T^kx\Vert+t\Vert x\Vert\Big) \notag \\
&\leq 3(4+12M)^mt^{\frac{k-n}{m-n}}\Vert x\Vert_{D(T^k)},\qquad 0<t\leq 1. \label{Eq_K_intermediate_4}
\end{align}
Hence, in \eqref{Eq_K_intermediate_3} and \eqref{Eq_K_intermediate_4} we have shown that in both cases $t\geq 1$ and $0<t\leq 1$ the estimate \eqref{Eq_K_intermediate} is satisfied.
\end{proof}

\bibliographystyle{amsplain}

\end{document}